\documentclass[a4paper,11pt]{amsart}
\usepackage{hyperref,latexsym}
\usepackage{enumerate}
\usepackage{graphicx}
\theoremstyle{plain}
\newtheorem{theorem}{Theorem}[section]
\newtheorem{lemma}[theorem]{Lemma}

\theoremstyle{definition}

\theoremstyle{remark}

\begin{document}

\title[Billiard characterization of spheres]
      {Billiard characterization of spheres}

\date{July 2018}
\author{Misha Bialy }
\address{M. Bialy, School of Mathematical Sciences, Tel Aviv
University, Israel} \email{bialy@post.tau.ac.il}

\thanks{Partially supported in part by the Israel Science Foundation grant
162/15}

\subjclass[2010]{} 

\begin{abstract} In this note we  study the higher dimensional convex billiards satisfying the so-called Gutkin property. A convex hypersurface $S$ satisfies this property if any chord $[p,q]$ which forms angle $\delta$ with the tangent hyperplane at $p$ has the same angle $\delta$ with the tangent hyperplane at $q$. Our main result is that the only convex hypersurface with this property in $\mathbf{R}^d, d\geq 3$ is a round sphere. 
	This extends previous results on Gutkin billiards obtained in \cite{B0}.
\end{abstract}

\maketitle

%%%%%%%%%%%%%%%%%%%%%%%%%%%%%%%%%%%%%%%%%%%%%%%%%%%%%%%%%%%%%%%%%%%%%%%%%%

%%%%%%%%%%%%%%%%%%%%%%%%%%%%%%%%%%%%%%%%%%%%%%%%%%%%%%%%%%%%%%%%%%%%%%%%%%

%%%%%%%%%%%%%%%%
\section{\bf Introduction and main result}
Consider a convex compact domain in Euclidean space $\mathbf{R}^{d}$
bounded by a smooth hypersurface $S$ with positive principal curvatures everywhere.
We shall call $S$ a Gutkin billiard table if
there exists $\delta\in (0;\pi/2)$ such that for any pair of points $p,q \in S $  the following condition is satisfied: if the angle between the vector $\overrightarrow {pq}$ with the tangent hyperplane to $S$ at $p$ equals $\delta$, then the angle between $\overrightarrow {pq}$
and the tangent hyperplane at $q$ also equals $\delta$. Note that the case $\delta=\pi/2$  is classical and corresponds to
bodies of constant width. So hereafter we will assume  $0<\delta<\pi/2$.

Planar billiard tables
with this property were found and explored by Eugene Gutkin
\cite{gu1},\cite{gu2} (see also \cite {T2}). He proved that planar domains with this
property different from round discs exist for
those values of $\delta$ which for some integer $n>3$ satisfy  the
equation
\begin {equation}\label{G}
\tan(n\delta)=n\tan \delta.
\end{equation}
Moreover, the shape of these domains is also very special.
It is an open conjecture by E.Gutkin that for any $\delta\in (0;\pi/2)$ no more than one integer $n>3$ can satisfy (\ref{G}).

It turns out that the property of equal angles becomes very rigid in higher dimensions. It is the aim of this note to prove the following.
\begin{theorem}\label{main}
The only Gutkin billiard tables in $\mathbf{R}^{d}, d\geq 3$ are round spheres.
\end{theorem}

Gutkin property of the hypersurface $S$ can be interpreted in terms of the billiard map. In these terms this property means the existence of an invariant hypersurface in the phase space of the billiard of very specific form (see Section 2). Another geometric situation leading to an invariant hypersurface in the phase space appears when there exists a convex caustic for the billiard. However the latter can exist only for ellipsoids, as was shown in \cite{b-g},\cite{ber}. It would be interesting to understand in more details the existence, geometric and dynamical properties of invariant hypersurfaces of billiard maps.

%Variational and dynamical characterization of circular billiard was given in %\cite{B1} and then in \cite{W}.
There are very few results on higher dimensional convex billiards.  In \cite{Sine} round spheres are characterized by the property that all the orbits of billiard are 2-planar. In \cite{B5} a variational study of orbits is proposed.  Periodic orbits of the higher dimensional billiards are studied in \cite{F-T}.
%%%%%%%%%%%%%%%%%%%%%%%%%%%%%%%%%%%%%%%%%%%%%%%%%%%%%%%%%%%%%%%%%%%%%%%%%

	Half of Theorem \ref{main} has been recently proven in \cite{B0}. In particular, it was shown there that the result holds true for $d=3$. 
	In this paper we complete the second half of the result. Thus in higher dimensions number theoretic properties of $\delta$ are irrelevant. 
	Our approach uses symplectic nature of the billiard ball map as well as geometry of convex bodies of constant width.
	
	In the Section 2 we recall the approach of \cite{B0} and summarize the needed material from \cite{B0}.
	In Section 3 we prove several lemmas, and Section 4 contains the proof of the main theorem.

\section {\bf Preliminaries and previous results} Proof of Theorem 1.1 requires symplectic and differential geometric properties of Gutkin billiards.
\subsection{Symplectic preliminaries} Consider Birkhoff billiard inside  hypersurface $S$. The phase space
$\Omega$ of the billiard consists of the set of oriented lines
intersecting $S$. The space of oriented lines in $\mathbf{R}^{d}$ is
isomorphic to $T^*\mathbf{S}^{d-1}$ and hence carries natural
symplectic structure. Birkhoff billiard map $\mathcal B$ acts on the space of
oriented lines and preserves this structure. Another way to describe
this symplectic structure is the following. Every oriented line $l$
intersecting $S$ at $p$ corresponds to a unit vector with foot point
$p$ on $S$. Orthogonal projection onto the tangent space $T_pS$ maps
hemisphere of  inward unit vectors with foot point $p$  onto unit
ball of the tangent space $ T_pS$ in 1-1 way. Thus the phase space of oriented
lines intersecting $S$ is isomorphic to unit (co-)ball bundle of
$S$. The canonical symplectic form of this bundle coincides with
that defined above. Here and below we identify co-vectors with
vectors by means of the scalar product induced from
$\mathbf{R}^{d}$.

 The hypersurface $S$ has Gutkin property with the angle
$\delta\in (0;\pi/2)$ if and only if the hypersurface  $\Sigma_\delta$ of the phase space $\Omega$ determined by the formula
$$
\Sigma_\delta = \{(p,v)\in \Omega: p\in S, v\in T_p S, |v|=\cos\delta\}
$$
is invariant under $\mathcal B$. As a corollary we get:
\begin{theorem}
	Let   $S\subset\mathbf{R}^{d}$ be a Gutkin billiard table.
	
	1. The billiard ball map $\mathcal B$ preserves characteristics of $\Sigma_\delta$. Moreover, $\mathcal B$
	 preserves the natural orientation of the characteristics.
	 
	 2. Characteristics of $\Sigma_\delta$
	 are geodesics of $S$ equipped with their tangent vectors of the constant length
	 $\cos\delta$. 
\end{theorem}

\begin{figure}[h]
\centering
\includegraphics[width=0.6\linewidth]{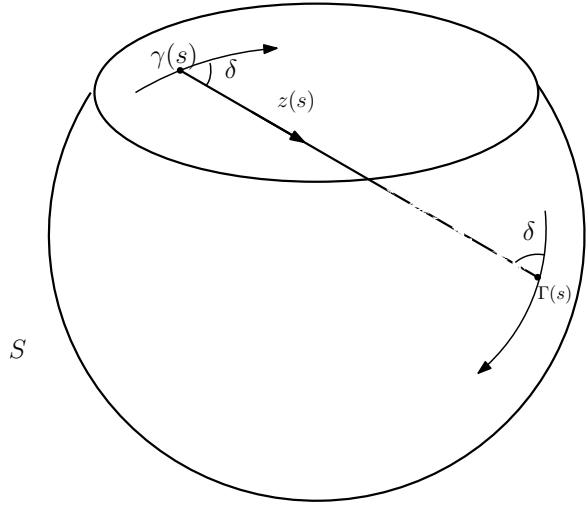}
\caption{Geodesic $\gamma$ is maped to $\Gamma$.}
\label{2}
\end{figure}

\subsection{Deviation from osculating 2-plane; planarity of geodesics}

Note, that since all principal curvatures of $S$ are assumed to be strictly positive, for any geodesic $\gamma$ on $S$
the curvature $k$ of $\gamma$ in $\mathbf{R}^{d}$ is strictly positive. Let us denote $v(s):=\dot{\gamma}(s)$, $n(s)$ the inner unit normal to $S$ at $\gamma(s)$. We can write first three Frenet formulae for a geodesic  $\gamma$ in  $\mathbf{R}^{d}$ as follows:
\begin{equation}\label{f1}
\dot{v}(s)=k(s)\ n(s).
\end{equation}

\begin{equation}\label{f2}
\dot{n}(s)=-k(s) v(s)+ \tau(s) w(s),
\end{equation}where $w$ is a unit vector in $\mathbf{R}^{d}$ orthogonal to $Span\{v,n\}$.
Also we have that $\dot w$ is orthogonal to $v$ and $w$ and we write :
\begin{equation}\label{f3}
\dot{w}(s)=-\tau(s)n(s)+\hat{w},
\end{equation}
where $\hat{w}$ is orthogonal to $Span\{v,n,w\}$.

If $d=3$, then $w$ is just a bi-normal vector of $\gamma$, $\hat{w}\equiv0$ and (\ref{f1}), (\ref{f2}),(\ref{f3}) are usual Frenet equations, where $\tau$ is torsion of $\gamma$.
It is important to note that also in higher dimensions one concludes from (\ref*{f2})  that the function $\tau$ vanishes if, and only if, the curve $\gamma$ lies in a 2-plane. Moreover, we have:

\begin{theorem}\label{tau}\
	
	1. Function $\tau$ satisfies linear differential equation:
	$$l\sin\delta (kl-\sin\delta)\dot{\tau}+B(s)\tau=0.$$
	
	2. The terms $(kl-\sin\delta)$ and $\tau$ do not vanish simultaneously.
	
3. If $\tau$ vanishes at one point it must vanish identically.
\end{theorem}

As a consequence of Theorem \ref{tau} we get planarity of some geodesic curves of $S$.
\begin{theorem}\label{planar}
	Every geodesic curve on $S$ which at some point $p$ passes in a principal direction lies necessarily in a 2-plane
	spanned by this direction and the normal line at $p$. Moreover, this geodesic curve
	has a principal direction at every point where it passes.
\end{theorem}

Using Theorem \ref{planar}, we get the following:

\begin{theorem}\label{width}Let $S$ be a convex hypersurface in $\mathbf{R}^{d}$ satisfying Gutkin property. Then:
	
	1. For $d=3$ it follows from  Theorem \ref{planar} that $S$ is a round sphere.
	
	2. For $d>3$,  hypersurface $S$ is of constant width. All geodesics passing in a principle direction are planer curves of the same constant width. Moreover, all geodesics passing through a point $p$ in  principle directions pass also through the antipodal point $\bar p$.
\end{theorem}

\section{\bf  Some lemmas}
\begin{lemma}\label {l0}Let $\gamma$ be a convex curve of constant width in the plane satisfying Gutkin property. Let $a,\bar a$
	be a pair of antipodal points on $\gamma$. Let $b,c$ be the points on $\gamma$ such that the chords $[a,b]$ and $[\bar a,c]$ form angle $\delta$ with $\gamma$ at both ends (see Fig.\ref{F1}).
	Then $c$ is the antipodal point of $b$, i.e. $c=\bar b$.
\end{lemma}
\begin{proof}Passing from $a$ to $b$ along $\gamma$ the tangent vector to $\gamma$ turns on the angle $2\delta$. 
	Analogously passing from $\bar a$ to $c$  the  tangent vector to $\gamma$ turns on $2\delta$. Together with the fact that $a, \bar a$ are antipodal we conclude that the tangent vectors to $\gamma$ at $b$ and at $c$ are parallel. Hence also the normals at these points. Hence it follows from double normal property of $\gamma$  (see \cite{BF}) that $c$ coincides with antipodal point $\bar b$.
\end{proof}
\begin{lemma}\label {l1}Let $\gamma$ be a convex curve of constant width $2R$ in the plane satisfying Gutkin property. Let $a,\bar a$
be a pair of antipodal points on $\gamma$. Then 
\begin{equation}\label{odin}
\rho(a)+\rho(\bar a)=2R;
\end{equation}
\begin{equation}\label{dva}
 \quad l+\bar l=4R\sin\delta.
\end{equation}
\end{lemma}
\begin{proof}First of the two equalities is obviously true for any convex curve of constant width. In order to prove the second  
	we choose the coordinate system centered at $a$ with $x$-axis tangent to $\gamma$ at $a$ and $y$-axis along the double normal $a\bar a$. We compute:
	$$
	l(\delta)=\frac{x(2\delta)}{\cos\delta}=\frac{1}{\cos \delta}\int_{0}^{2\delta}\cos\phi\ \rho(\phi)d\phi,
	$$
	where $\rho(\phi)$ is curvature radius as function of the angle $\phi $ between tangent vector to $\gamma$ and the $x$-axis.
	Similarly we have:
	$$
		\bar l(\delta)=\frac{-x(\pi+2\delta)}{\cos\delta}=-\frac{1}{\cos \delta}\int_{\pi}^{\pi+2\delta}\cos\phi\ \rho(\phi)d\phi=
	$$
	$$
	=\frac{1}{\cos \delta}\int_{0}^{2\delta}\cos\phi\ \rho(\pi+\phi)d\phi.
	$$
	Adding up these two formulas, we get:
	
$$
l+\bar l=\frac{1}{\cos \delta}\int_{0}^{2\delta}\cos\phi(\rho(\phi)+\rho(\bar\phi)d\phi=\frac{1}{\cos \delta}\int_{0}^{2\delta}\cos\phi\cdot 2R\ d\phi=4R\sin\delta.
$$
	\end{proof}
\begin{figure}[h]
	\centering
	\includegraphics[width=0.6\linewidth]{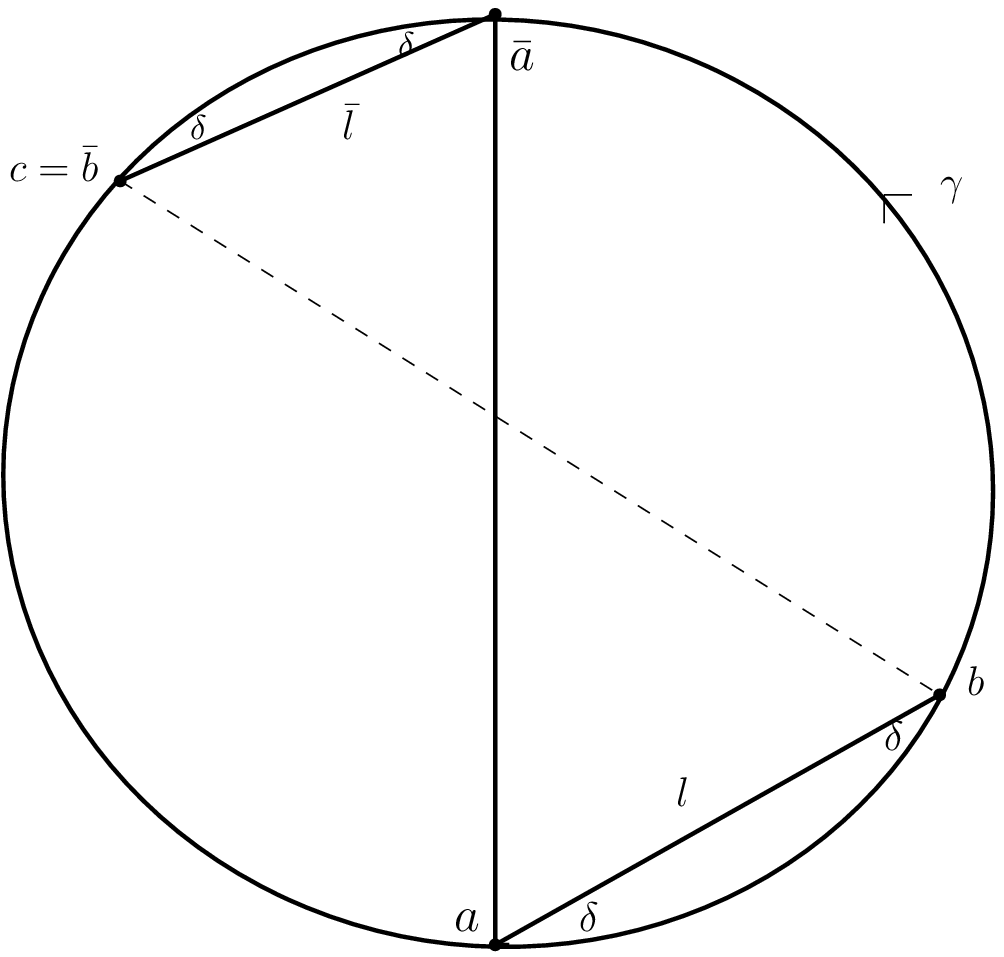}
	\caption{}
	\label{F1}
\end{figure}
\begin{lemma}\label{l2}
Let $\gamma$ be a convex curve of constant width $2R$ in the plane satisfying Gutkin property. It then follows that for any point of $\gamma$ the following inequality is true:
\begin{equation}\label{tri}
kl>\sin\delta.
\end{equation}
Moreover, for any pair of antipodal points $a,\bar a$ we have the following alternative:
either the inequality
\begin{equation}\label{chetyre}
kl<2\sin\delta.
\end{equation}
holds for exactly one of the points $a$ or $\bar a$ and the opposite inequality holds for the other one,
or the equality
$$
kl=2\sin\delta
$$
holds for both points $a$ and $\bar a$.

\end{lemma}
\begin{proof}
	 Inequality (\ref{tri}) follows from two facts. 
	
	The first fact is that, since $\gamma$
	satisfies Gutkin property, the mapping 
	$$
	\gamma(s)\rightarrow \gamma(s)+l(s)(\cos\delta\ \dot{\gamma}(s)+\sin\delta\  n(\gamma(s)))
	$$
is a diffeomorpphism and hence computing the derivative we get
$$
k(s)l(s)-\sin\delta\neq 0, \forall s.
$$	The same conclusion can be deduced from the statement 2. of Theorem \ref*{tau}.
Let $a$ be a point on the curve $\gamma$ of minimal curvature radius.
Then the osculating circle is contained entirely inside $\gamma$ and hence at this point (see Fig.\ref{F2})
$$l>|a-p|=2\rho_{min}\sin\delta>\rho_{min}\sin\delta.$$
But then $
kl>\sin\delta
$ at every point. 

\begin{figure}[h]
	\centering
	\includegraphics[width=0.6\linewidth]{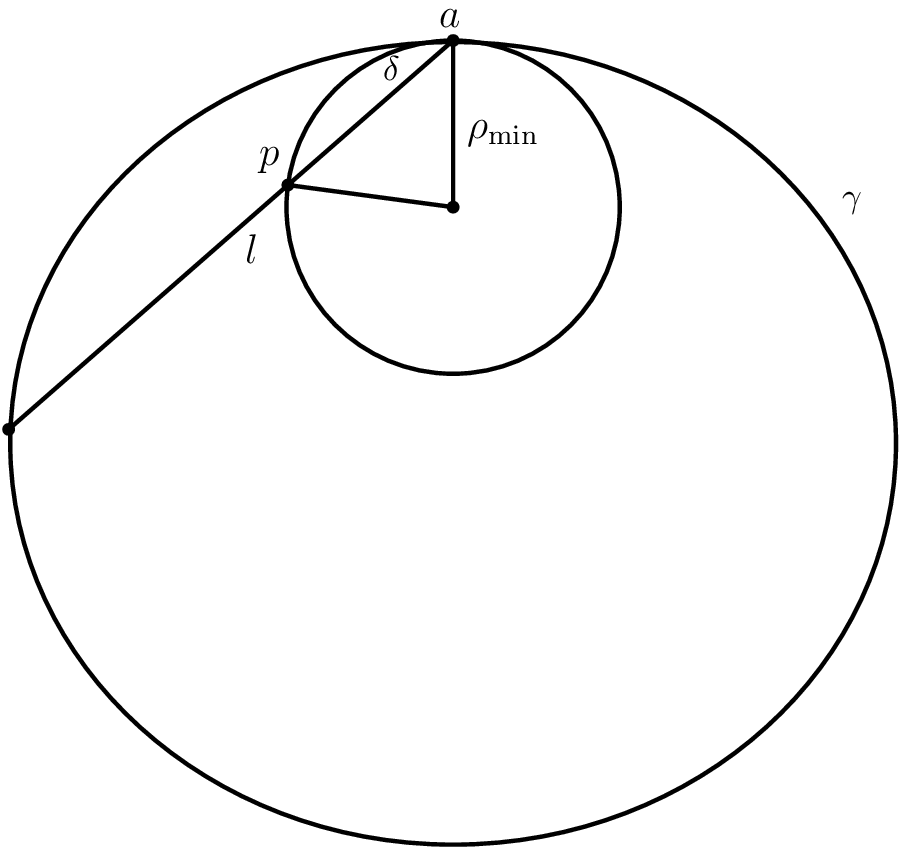}
	\caption{}
	\label{F2}
\end{figure}

In order to prove the second claim assume that for a point $a$
$$l>2\rho\sin\delta.$$
Then we have by Lemma \ref{l1}:
$$
\bar l=4R\sin\delta -l<4R\sin\delta-2\rho\sin\delta=2\bar\rho\sin\delta.
$$
This completes the proof of the Lemma.

	\end{proof}

Let $a\in S$ be any point. Chose any two orthogonal principle directions $v_1, v_2$ at $a$. Let us denote $\gamma_1,\gamma_2$ the corresponding geodesics lying in the 2-planes $\sigma_{1},\sigma_{2}$ which are of constant width and satisfy Gutkin property (Theorem \ref{width}).
We denote by $E$ the 3-space containing them. We shall denote by $v_i$, and $n_i, i=1,2$
tangent vectors and inner normals to $\gamma_i$. Since $\gamma_i$ are geodesics, $n_i$ are normals to the hypersurface $S$. We choose the arc-length parameters 
$t,s$ along $\gamma_1,\gamma_2$ respectively, so that
$$
\gamma_1(0)=\gamma_2(0)=a.
$$ We denote $k_i, i=1,2$ the curvature functions.

\begin{lemma}\label{l3}Let $S$ be a convex hypersurface in $\mathbf{R}^d$ satisfying Gutkin property. Let $\gamma_1,\gamma_2$ be the geodesics, as above, lying in orthogonal 2-planes $\sigma_1,\sigma_2$ (Fig.\ref{F3}). Then the curvature $k_2(a)$ satisfies the following quadratic equation: 
	\begin{equation}\label{l}
	Ax^2+Bx+C=0, 
	\end{equation} where $A,B$ and $C$ depend only on $\gamma_1$:
	$$
	A=l_1\sin \delta(k_1(b)l_1-\sin\delta),\quad B=2\sin\delta-k_1(b)l_1(1+\sin^2\delta),
	$$
	$$
	C=\frac{\sin\delta}{l_1}(k_1(b)l_1-2\sin\delta),
	$$ and $l_1$ is the chord of $\gamma_1$ starting at with the angle $\delta$ at $a$ and ending at $b$.
	\end{lemma}
\begin{figure}[h]
	\centering
	\includegraphics[width=0.55\linewidth]{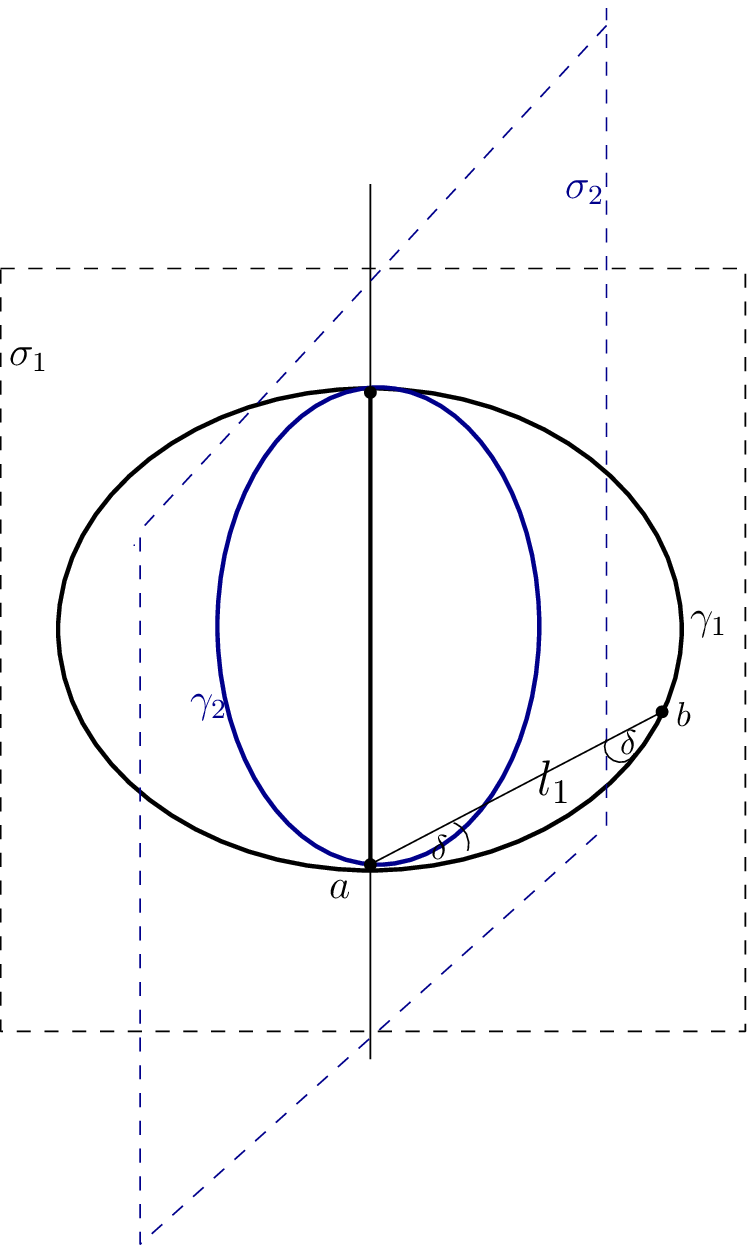}
	\caption{}
	\label{F3}
\end{figure}
\begin{proof}
	
The idea of the proof is as follows. For the initial moment $s=0$ we have the chord $[a,b]$ (Fig. \ref{F3}). We start moving the end $a$ of the chord along $\gamma_2$ to the point $\gamma_2(s)$ while the other end of the chord remains on $\gamma_1$. To describe this movement, let us consider the cone with the vertex at the point $\gamma_2(s)$ with the axis $n_2(s)$ and  the angle $(\pi/2-\delta)$ at the vertex. For the initial moment $s=0$ this cone intersects $\gamma_1$ transversally at the point $b$. Therefore also for small $s$ the cone intersects $\gamma_1$ at a point $\gamma_1(t(s))$, where $t(s)$ is a smooth function by transversality. By Gutkin property, the chord must have the same angle $(\pi/2-\delta)$ with the normal $n_1(t(s))$ also at the second end.
So we have two identities
\begin{equation}\label{I1}
\left< n_2(s), \gamma_1(t(s))-\gamma_2(s) \right>=
l(s)\sin\delta;
\end{equation}
\begin{equation}\label{I2}
\left< n_1(t(s)), \gamma_1(t(s))-\gamma_2(s) \right>=-
l(s)\sin\delta,
\end{equation}
where
$$
l(s):=|\gamma_1(t(s))-\gamma_2(s)|
$$
The next step is to differentiate these equalities twice with respect to $s$ at $s=0$. The second derivatives of both identities contain $t^{''}(0)$, equating the expressions for $t^{''}(0)$ from the first and the second gives the needed quadratic equation. 

In more details this step goes as follows. Let us note that in the computations below we consistently use the configuration for $s=0$ and Frenet formulas.

Differentiating once: 

\begin{equation}\label{e1}
\left<-k_2\gamma_2^{'}, \gamma_1-\gamma_2\right>+\left<n_2, \gamma_1\cdot t^{'}-\gamma_2^{'}\right>={\frac{\sin\delta}{l}}\left<\dot\gamma_1\cdot t^{'}-\gamma_2^{'},\gamma_1-\gamma_2\right>;
\end{equation}\begin{equation}\label{e2}
\left<-k_1\dot {\gamma_1}\cdot t^{'}, \gamma_1-\gamma_2\right>+\left<n_1, \gamma_1\cdot t^{'}-\gamma_2^{'}\right>=-{\frac{\sin\delta}{l}}\left<\dot\gamma_1\cdot t^{'}-\gamma_2^{'},\gamma_1-\gamma_2\right>.
\end{equation}
We use dot and prime for differentiation with respect to $t$ and $s$ respectively.
At this stage note that for $s=0$ one gets from (\ref{e1})
$$
t^{'}(0)\sin 2\delta=t^{'}(0)\sin\delta\cos\delta,
$$
and so 
\begin{equation}\label{e3}
t^{'}(0)=0.
\end{equation}
Also $$
l^{'}=\frac{1}{l}\left<\dot\gamma_1\cdot t^{'}-\gamma_2^{'},\gamma_1-\gamma_2\right>\quad \Rightarrow\  l^{'}(0)=0.
$$

Differentiating (\ref{e1}) and evaluating at $s=0$ one has after easy computations:
$$
-k_2^2(a)l_1\sin\delta+k_2(a)+t^{''}\sin 2\delta=\sin \delta\left(t^{''}\cos-k_2(a)\sin\delta+\frac{1}{l_1}\right).
$$
This gives 
\begin{equation}\label{e4}
t^{''}=-\frac{k_2(a)(1+\sin^2\delta)-\frac{\sin\delta}{l_1}-k_2^2(a)l_1\sin\delta}{\sin\delta\cos\delta}.
\end{equation}
Analogously differentiating (\ref{e2}) and evaluating at $s=0$ we get:
\begin{equation}\label{e5}
t^{''}=\frac{k_2(a)\cos^2\delta-\frac{\sin\delta}{l_1}}{\sin\delta\cos\delta-k_1(b)l_1\cos\delta}.
\end{equation}
From (\ref*{e4}), (\ref*{e5}) collecting terms  we get precisely equation (\ref {l}), proving the Lemma.
	\end{proof}

\section{\bf Proof of Theorem \ref{main}}
	Since the case of $d=3$ was considered in \cite{B0}, we shall assume here that $d>3$.
	Let $a$ be an arbitrary point of the hypersurface $S$. To prove Theorem \ref{main} it is sufficient to prove that every $a$ is 
	totally umbilic point, i.e.  all principle curvatures at $a$ are equal. 
	Choose an orthonormal basis $\{v_1, v_2.,..,v_{d-1}\}$ in $T_aS$ consisting of principle directions.
	Let us consider the geodesic curves $\gamma_1,\gamma_2,..,\gamma_{d-1}$ in the directions $\{v_1, v_2.,..,v_{d-1}\}$. These are plane curves intersecting in the antipodal point $\bar a$ (Theorem \ref*{width}).
	Let $k_1, k_2,...,k_{d-1}$ be the curvature functions of $\gamma_1,\gamma_2,..,\gamma_{d-1}$.
In order to prove that $a$ is totally umbilic we prove below the following claim:  $d-2$ principle curvatures $k_2,...,k_{d-1}$ are all equal. Choosing  $\gamma_{d-1}$ instead of $\gamma_1$ and applying the claim we get 
$$
k_1(a)=...=k_{d-2}(a).
$$
Thus using the fact that $d\geq3$ we conclude that $a$ is a totally umbilic point of $S$.

We turn now to the proof of the claim.
We shall consider two cases:
\vskip3mm

\underline{Case 1}. Suppose that at the point $b$ (see Fig. \ref{F3}) the inequality $$k_1(b)l_1-2\sin\delta\leq0$$ is valid.
In this case the coefficients of the quadratic equation (\ref{l}) satisfy by Lemma \ref{l2} $$A>0, C\leq 0.$$
By Lemma \ref{l3} we have in this case that for $k_2(a)$ there is only one possible value, namely the only positive root $r$ of equation (\ref{l}). Note that the coefficients $A,B,C$ and hence the positive root $r$ of equation (\ref{l}) depend only on $\gamma_1$. Therefore, replacing $\gamma_2$ by any of the geodesics $\gamma_i, i=3,..,d-1$  we get that all curvatures $$k_2(a)=...=k_{d-1}(a)=r$$ are equal to the positive root of equation (\ref{l}). 
\vskip3mm

 \underline{Case 2}. Suppose now that $$k_1(b)l_1-2\sin\delta>0.$$ In this case the previous argument does not work for the point $a$ because in this case equation ($\ref{l}$) has two positive roots.
However, we can apply the previous reasoning for the antipodal points $\bar a,\bar b$. Indeed, it follows from Lemma \ref{l2} that in 
Case 2
$$
k_1(\bar b)\bar l_1-2\sin\delta< 0,
$$
and so, according to the proof given for the Case 1, we have for $\bar a$:
$$
k_1(\bar a)=...=k_{d-2}(\bar a).
$$
But then  the equalities 
$$k_2(a)=...=k_{d-1}(a)$$
  hold true 
  also for point $a$,
  by the relation (\ref{odin}) of principle curvature radii at the antipodal points.
This completes the proof of Theorem \ref{main} for $d> 3$.
%%%%%%%%%%%%%%%%%%%%%%%%%%%%%%%%%%%%%%%%

\section*{Acknowledgments}
%It is a pleasure to thank  Boris Khesin, Yuri Suris and Sergei Tabachnikov for stimulating discussions.
This research was supported in part by ISF grant 162/15.
It is a pleasure to thank Yurii Dmitrievich Burago for useful consultations.

%%%%%%%%%%%%%%%%%%%%%%%%%%%%%%%%%%%%%%%%%%%%%%%%%%%%%%%%


\begin{thebibliography}{10}
	\bibitem{B0} Bialy, M. {\it Gutkin billiard tables in higher dimensions and rigidity}, accepted to Nonlinearity.

	\bibitem{B5}
	\newblock Bialy, M.
	\newblock \emph{Maximizing orbits for higher dimensional convex
		billiards},
	\newblock Journal of Modern Dynamics, \textbf{3},no.1 (2009),
	51--59.
	%\bibitem {B2}Bialy, M. Maximizing orbits for higher dimensional convex
	%billiards.\\ Journal of Modern Dynamics V.3, N1 (2009) 51--59.
	

	
	\bibitem {ber}Berger, M., \emph{Sur les caustiques de surfaces en dimension 3}, C.R. Acad. Sci. Paris, Ser. I Math. \textbf{311} (1990) 333--336.
	
	\bibitem{BF} Bonnesen, T.; Fenchel, W. {\it Theory of convex bodies}. Translated from the German and edited by L. Boron, C. Christenson and B. Smith. BCS Associates, Moscow, ID, 1987.
	
	\bibitem {b-g}Gruber, P. {\it Only ellipsoids have caustics.} Math. Ann. \textbf{303} (1995), no. 2, 185–-194.
	
	
	\bibitem {gu1}Gutkin, E. {\it Capillary floating and the billiard ball problem.} J. Math. Fluid Mech. \textbf{14}, 363–-382 (2012).
	
	\bibitem{gu2} Gutkin, E.  {\it Addendum to: Capillary floating and the billiard ball problem.} J. Math. Fluid Mech. \textbf{15} (2013), no. 2, 425-–430.
	
	\bibitem {F-T}Farber, M.; Tabachnikov, S. {\it Topology of cyclic configuration spaces and periodic trajectories of multi-dimensional billiards}. Topology 41 (2002), no. 3, 553-–589.

	\bibitem{Sine} Sine, R. {\it A characterization of the ball in $\mathbf R^3$}. Amer. Math. Monthly 83 (1976), no. 4, 260–-261.
	\bibitem{T2} S. Tabachnikov, {\it Billiards.} Panor. Synth. No. 1 (1995), 142 pp.
	
	
\end{thebibliography}
\end{document}